\documentclass[reqno]{amsart}
\usepackage{amssymb,amsmath,hyperref}
\usepackage{amsrefs}
\usepackage[foot]{amsaddr}
\usepackage{bbold,stackrel}
 
\newtheorem{thm}{Theorem}

\newtheorem{cor}[thm]{Corollary}
\newtheorem{defi}[thm]{Definition}
\newtheorem{rem}[thm]{Remark}
\newtheorem{nota}[thm]{Notation}

\newtheorem{ack}[thm]{Acknowledgement}

\newtheorem*{tempo*}{Template}

\newcommand\be{\begin{equation}}
\newcommand\ee{\end{equation}} 

\usepackage{amsmath,amsfonts} 

\usepackage[applemac]{inputenc}

\def\bdefi{\begin{defi}\rm}
\def\edefi{\end{defi}}
\def\bnota{\begin{nota}\rm}
\def\enota{\end{nota}}
\def\FIVE{\Pi_{1}^{1}\text{-\textup{\textsf{CA}}}_{0}}

\def\SIX{\Pi_{2}^{1}\text{-\textsf{\textup{CA}}}_{0}}

\def\SIXK{\Pi_{k}^{1}\text{-\textsf{\textup{CA}}}_{0}^{\omega}}

\def\ATR{\textup{\textsf{ATR}}}

\def\Z{\textup{\textsf{Z}}}

\def\NFP{\textup{\textsf{NFP}}}
\def\ZFC{\textup{\textsf{ZFC}}}

\def\ZF{\textup{\textsf{ZF}}}

 \def\r{\mathbb{r}}

\def\c{\textup{\textsf{c}}}
\def\RCA{\textup{\textsf{RCA}}}
\def\({\textup{(}}
\def\){\textup{)}}

\def\RCAo{\textup{\textsf{RCA}}_{0}^{\omega}}
\def\ACAo{\textup{\textsf{ACA}}_{0}^{\omega}}

\def\WKL{\textup{\textsf{WKL}}}

\def\bye{\end{document}}
\def\N{{\mathbb  N}}
\def\Q{{\mathbb  Q}}
\def\R{{\mathbb  R}}
\def\L{\textsf{\textup{L}}}

\def\ind{\textup{{ind }}}
\def\Ind{\textup{{Ind }}}

\def\di{\rightarrow}
\def\asa{\leftrightarrow}
\def\ACA{\textup{\textsf{ACA}}}
\def\PUNI{\textup{\textsf{PUNI}}}

\def\QFAC{\textup{\textsf{QF-AC}}}

\def\HBU{\textup{\textsf{HBU}}}
\def\PRA{\textup{\textsf{PRA}}}
\def\IND{\textup{\textsf{IND}}}
\def\HBT{\textup{\textsf{HBT}}}

\def\LIN{\textup{\textsf{LIN}}}
\def\LIL{\textup{\textsf{LIL}}}

\def\ECF{\textup{\textsf{ECF}}}

\def\SCF{\textup{\textsf{SCF}}}

\usepackage{graphicx}
\usepackage{tikz, float}
\usetikzlibrary{matrix, shapes.misc}

\numberwithin{equation}{section}
\numberwithin{thm}{section}

\usepackage{comment}

\begin{document}
\title[Reverse Mathematics of topology]{Reverse Mathematics of topology \\ {\tiny dimension, paracompactness, and splittings}}
\author{Sam Sanders}
\address{School of Mathematics, University of Leeds \& Department of Mathematics, TU Darmstadt, Germany}
\email{sasander@me.com}


\begin{abstract}
Reverse Mathematics (RM hereafter) is a program in the foundations of mathematics founded by Friedman and developed extensively by Simpson and others.  
The aim of RM is to find the minimal axioms needed to prove a theorem of ordinary, i.e.\ non-set-theoretic, mathematics.  As suggested by the title, 
this paper deals with the study of the topological notions of \emph{dimension} and \emph{paracompactness}, inside Kohlenbach's \emph{higher-order} RM. 
As to \emph{splittings}, there are some examples in RM of theorems $A, B, C$ such that $A\asa (B\wedge C)$, i.e.\ $A$ can be \emph{split} into two independent (fairly natural) parts $B$ and $C$, and the aforementioned topological notions give rise to a number of splittings involving \emph{highly natural} $A, B, C$.  
Nonetheless, the higher-order picture is markedly different from the second-one: in terms of comprehension axioms, the proof in higher-order RM of e.g.\ the paracompactness of the unit interval requires \emph{full second-order arithmetic}, while the second-order/countable version of paracompactness of the unit interval is provable in the base theory $\RCA_{0}$. 
We obtain similarly `exceptional' results for the \emph{Urysohn identity}, the \emph{Lindel\"of lemma}, and \emph{partitions of unity}.  We show that our results exhibit a certain \emph{robustness}, in that they do not depend on the exact definition of cover, even in the absence of the axiom of choice.    
\end{abstract}
%

\maketitle
\thispagestyle{empty}


\section{Introduction}\label{intro}
Reverse Mathematics (RM hereafter) is a program in the foundations of mathematics initiated around 1975 by Friedman (\cites{fried,fried2}) and developed extensively by Simpson (\cite{simpson2}) and others.  
We refer to \cite{stillebron} for a basic introduction to RM and to \cite{simpson2, simpson1} for an (updated) overview of RM.  We will assume basic familiarity with RM, the associated `Big Five' systems and the `RM zoo' (\cite{damirzoo}).  
We do introduce Kohlenbach's \emph{higher-order} RM in some detail Section \ref{HORM}.    

\smallskip

Topology studies those properties of space that are invariant under continuous deformations.
The modern subject was started by Poincar\'e's \emph{Analysis Situs} at the end of 19th century, 
and rapid breakthroughs were established by Brouwer in a two-year period starting 1910, as discussed in \cite{godsgeschenk}*{p.\ 168}. 
We generally base ourselves on the standard monograph by Munkres (\cite{munkies}).

\smallskip

Now, the RM of topology has been studied inside the framework of second-order arithmetic in e.g.\ \cite{mummymf, mummyphd, mummy}.
This approach makes heavy use of \emph{coding} to represent uncountable objects via countable approximations.  Hunter develops the higher-order RM of topology in \cite{hunterphd}, and points out some 
potential problems with the aforementioned coding practice.  Hunter's observations constitute our starting point and motivation: 
working in higher-order RM, we study the RM of notions like \emph{dimension} and \emph{paracompactness} motivated as follows: the former is among the most basic/fundamental notions of topology, while the latter has already been studied in second-order RM, e.g.\ in the context of metrisation theorems (\cite{simpson2, mummymf}). 

\smallskip

As it turns out, the picture we obtain in higher-order RM is \emph{completely different} from the well-known picture in second-order RM.  
For instance, in terms of comprehension axioms, the proof in higher-order RM of the paracompactness of the unit interval requires \emph{full second-order arithmetic} by Theorem \ref{paramaeremki}, 
while the second-order/countable version of paracompactness of the unit interval is provable in the base theory $\RCA_{0}$ of second-order RM by \cite{simpson2}*{II.7.2}. 
Furthermore, the \emph{Urysohn identity} connects various notions of dimension, and a proof of this identity {for $[0,1]$} similarly requires (comprehension axioms as strong as) full second-order arithmetic.  
We also study the \emph{Lindel\"of lemma} and \emph{partitions of unity}.

\smallskip

The aforementioned major difference between second-order and higher-order RM begs the question as to how robust the results in this paper are.  
For instance, do our theorems depend on the exact definition of cover?  What happens if we adopt a more general definition?  We show in Sections \ref{introke}, \ref{kerkend}, and \ref{BBB}
that our results indeed boast a lot of robustness, and in particular that they do not depend on the definition of cover, even in the absence of the axiom of (countable) choice. 
The latter feature is important in view of the topological `disasters' (see e.g.\ \cite{kermend}) that manifest themselves in the absence of the axiom of (countable) choice.  
A rather elegant base theory is formulated in Section \ref{BBB}, based on the \emph{neighbourhood function principle} from \cite{troeleke1}.

\smallskip

We also obtain a number of highly natural \emph{splittings}, where the latter is defined as follows.  
As discussed in e.g.\ \cite{dsliceke}*{\S6.4}, there are (some) theorems $A, B, C$ in the RM zoo such that $A\asa (B\wedge C)$, i.e.\ $A$ can be \emph{split} into two independent (fairly natural) parts $B$ and $C$ (over $\RCA_{0}$).  
It is fair to say that there are only few \emph{natural} examples of splittings in second-order RM, though such claims are invariably subjective in nature. 
A large number of splittings in \emph{higher-order} RM may be found in \cite{samsplit}.  
%
%
%

\smallskip

Finally, like in \cite{dagsamIII, dagsamV}, statements of the form `a proof of this theorem requires full second-order arithmetic' should be interpreted in reference to the usual scale of comprehension axioms that is part of the \emph{G\"odel hierarchy} (see Section \ref{kurtzenhier} for the latter).  
The previous statement thus (merely) expresses that there is no proof of this theorem using comprehension axioms restricted to a sub-class, like e.g.\ $\Pi_{k}^{1}$-formulas (with only first and second-order parameters).  An intuitive visual clarification may be found in 
Figure \ref{xxy}, where the statement \emph{the unit interval is paracompact} is shown to be independent of the medium range of the G\"odel hierarchy.  Similarly, when we say `provable without the axiom of choice', we ignore the use of the very weak instances of the latter included in the base theory of higher-order RM.

\smallskip

In conclusion, it goes without saying that our results highlight a \emph{major} difference between second- and higher-order arithmetic, and the associated development of RM.  We leave it the reader to draw conclusions from this observation.

\section{Preliminaries}\label{preli}
\subsection{Higher-order Reverse Mathematics}\label{HORM}
We sketch Kohlenbach's \emph{higher-order Reverse Mathematics} as introduced in \cite{kohlenbach2}.  In contrast to `classical' RM, higher-order RM makes use of the much richer language of \emph{higher-order arithmetic}.  

\smallskip

As suggested by its name, {higher-order arithmetic} extends second-order arithmetic $\Z_{2}$.  Indeed, while the latter is restricted to numbers and sets of numbers, higher-order arithmetic also has sets of sets of numbers, sets of sets of sets of numbers, et cetera.  
To formalise this idea, we introduce the collection of \emph{all finite types} $\mathbf{T}$, defined by the two clauses:
\begin{center}
(i) $0\in \mathbf{T}$   and   (ii)  If $\sigma, \tau\in \mathbf{T}$ then $( \sigma \di \tau) \in \mathbf{T}$,
\end{center}
where $0$ is the type of natural numbers, and $\sigma\di \tau$ is the type of mappings from objects of type $\sigma$ to objects of type $\tau$.
In this way, $1\equiv 0\di 0$ is the type of functions from numbers to numbers, and where  $n+1\equiv n\di 0$.  We also write $\mathbb{1}$ for the type $1\di 1$.   Viewing sets as given by characteristic functions, we note that $\Z_{2}$ only includes objects of type $0$ and $1$, i.e.\ natural numbers and sets thereof.    

\smallskip

The language $\L_{\omega}$ includes variables $x^{\rho}, y^{\rho}, z^{\rho},\dots$ of any finite type $\rho\in \mathbf{T}$.  Types may be omitted when they can be inferred from context.  
The constants of $\L_{\omega}$ includes the type $0$ objects $0, 1$ and $ <_{0}, +_{0}, \times_{0},=_{0}$  which are intended to have their usual meaning as operations on $\N$.
Equality at higher types is defined in terms of `$=_{0}$' as follows: for any objects $x^{\tau}, y^{\tau}$, we have
\be\label{aparth}
[x=_{\tau}y] \equiv (\forall z_{1}^{\tau_{1}}\dots z_{k}^{\tau_{k}})[xz_{1}\dots z_{k}=_{0}yz_{1}\dots z_{k}],
\ee
if the type $\tau$ is composed as $\tau\equiv(\tau_{1}\di \dots\di \tau_{k}\di 0)$.  
Furthermore, $\L_{\omega}$ also includes the \emph{recursor constant} $\mathbf{R}_{\sigma}$ for any $\sigma\in \mathbf{T}$, which allows for iteration on type $\sigma$-objects as in the special case \eqref{special}.  
Formulas and terms are defined as usual.  
\bdefi The base theory $\RCAo$ consists of the following axioms:
\begin{enumerate}
\item  Basic axioms expressing that $0, 1, <_{0}, +_{0}, \times_{0}$ form an ordered semi-ring with equality $=_{0}$.
\item Basic axioms defining the well-known $\Pi$ and $\Sigma$ combinators (aka $K$ and $S$ in \cite{avi2}), which allow for the definition of \emph{$\lambda$-abstraction}. 
\item The defining axiom of the recursor constant $\mathbf{R}_{0}$: For $m^{0}$ and $f^{1}$: 
\be\label{special}
\mathbf{R}_{0}(f, m, 0):= m \textup{ and } \mathbf{R}_{0}(f, m, n+1):= f( n,\mathbf{R}_{0}(f, m, n)).
\ee
\item The \emph{axiom of extensionality}: for all $\rho, \tau\in \mathbf{T}$, we have:
\be\label{EXT}\tag{$\textsf{\textup{E}}_{\rho, \tau}$}  
(\forall  x^{\rho},y^{\rho}, \varphi^{\rho\di \tau}) \big[x=_{\rho} y \di \varphi(x)=_{\tau}\varphi(y)   \big].
\ee 
\item The induction axiom for quantifier-free\footnote{To be absolutely clear, variables (of any finite type) are allowed in quantifier-free formulas of the language $\L_{\omega}$: only quantifiers are banned.} formulas of $\L_{\omega}$.
\item $\QFAC^{1,0}$: The quantifier-free axiom of choice as in Definition \ref{QFAC}.
\end{enumerate}
\edefi
\bdefi\label{QFAC} The axiom $\QFAC$ consists of the following for all $\sigma, \tau \in \textbf{T}$:
\be\tag{$\QFAC^{\sigma,\tau}$}
(\forall x^{\sigma})(\exists y^{\tau})A(x, y)\di (\exists Y^{\sigma\di \tau})(\forall x^{\sigma})A(x, Y(x)),
\ee
for any quantifier-free formula $A$ in the language of $\L_{\omega}$.
\edefi
As discussed in \cite{kohlenbach2}*{\S2}, $\RCAo$ and $\RCA_{0}$ prove the same sentences `up to language' as the latter is set-based and the former function-based.  Recursion as in \eqref{special} is called \emph{primitive recursion}; the class of functionals obtained from $\mathbf{R}_{\rho}$ for all $\rho \in \mathbf{T}$ is called \emph{G\"odel's system $T$} of all (higher-order) primitive recursive functionals.  

\smallskip

We use the usual notations for natural, rational, and real numbers, and the associated functions, as introduced in \cite{kohlenbach2}*{p.\ 288-289}.  
\begin{defi}[Real numbers and related notions in $\RCAo$]\label{keepintireal}\rm~
\begin{enumerate}
\item Natural numbers correspond to type zero objects, and we use `$n^{0}$' and `$n\in \N$' interchangeably.  Rational numbers are defined as signed quotients of natural numbers, and `$q\in \Q$' and `$<_{\Q}$' have their usual meaning.    
\item Real numbers are coded by fast-converging Cauchy sequences $q_{(\cdot)}:\N\di \Q$, i.e.\  such that $(\forall n^{0}, i^{0})(|q_{n}-q_{n+i})|<_{\Q} \frac{1}{2^{n}})$.  
We use Kohlenbach's `hat function' from \cite{kohlenbach2}*{p.\ 289} to guarantee that every $q^{1}$ defines a real number.  
\item We write `$x\in \R$' to express that $x^{1}:=(q^{1}_{(\cdot)})$ represents a real as in the previous item and write $[x](k):=q_{k}$ for the $k$-th approximation of $x$.    
\item Two reals $x, y$ represented by $q_{(\cdot)}$ and $r_{(\cdot)}$ are \emph{equal}, denoted $x=_{\R}y$, if $(\forall n^{0})(|q_{n}-r_{n}|\leq {2^{-n+1}})$. Inequality `$<_{\R}$' is defined similarly.  
We sometimes omit the subscript `$\R$' if it is clear from context.           
\item Functions $F:\R\di \R$ are represented by $\Phi^{1\di 1}$ mapping equal reals to equal reals, i.e. $(\forall x , y\in \R)(x=_{\R}y\di \Phi(x)=_{\R}\Phi(y))$.
\item The relation `$x\leq_{\tau}y$' is defined as in \eqref{aparth} but with `$\leq_{0}$' instead of `$=_{0}$'.  Binary sequences are denoted `$f^{1}, g^{1}\leq_{1}1$', but also `$f,g\in C$' or `$f, g\in 2^{\N}$'.  
\end{enumerate}
\end{defi}
Finally, we mention the $\ECF$-interpretation, of which the technical definition may be found in \cite{troelstra1}*{p.\ 138, 2.6}.
Intuitively speaking, the $\ECF$-interpretation $[A]_{\ECF}$ of a formula $A\in \L_{\omega}$ is just $A$ with all variables 
of type two and higher replaced by countable representations of continuous functionals. 
The $\ECF$-interpretation connects $\RCAo$ and $\RCA_{0}$ (see \cite{kohlenbach2}*{Prop.\ 3.1}) in that if $\RCAo$ proves $A$, then $\RCA_{0}$ proves $[A]_{\ECF}$, again `up to language', as $\RCA_{0}$ is 
formulated using sets, and $[A]_{\ECF}$ is formulated using types, namely only using type zero and one objects.  

\subsection{Some axioms of higher-order arithmetic}\label{saxioms}
We introduce some functionals from \cite{kohlenbach2,dagsamIII, dagsamV} which give rise to the higher-order counterparts of $\Z_{2}$, and some of the Big Five systems.  
In each case, the higher-order system is a conservative extension of the second-order system, for a fairly broad formula class. 

\smallskip
\noindent
First of all, $\ACA_{0}$ is readily derived from the following `Turing jump' functional:
\be\label{muk}\tag{$\exists^{2}$}
(\exists \varphi^{2}\leq_{2}1)(\forall f^{1})\big[(\exists n)(f(n)=0) \asa \varphi(f)=0    \big]. 
\ee
and $\ACA_{0}^{\omega}\equiv\RCAo+(\exists^{2})$ proves the same $\Pi_{2}^{1}$-sentences as $\ACA_{0}$ by \cite{yamayamaharehare}*{Theorem~2.2}. 
This functional is \emph{discontinuous} at $f=_{1}11\dots$, and $(\exists^{2})$ is equivalent to the existence of $F:\R\di\R$ such that $F(x)=1$ if $x>_{\R}0$, and $0$ otherwise (\cite{kohlenbach2}*{\S3}).  

\smallskip
\noindent
Secondly, $\FIVE$ is readily derived from the following `Suslin functional':
\be\tag{$S^{2}$}
(\exists S^{2}\leq_{2}1)(\forall f^{1})\big[  (\exists g^{1})(\forall x^{0})(f(\overline{g}n)=0)\asa S(f)=0  \big], 
\ee
and $\FIVE^{\omega}\equiv \RCAo+(S^{2})$ proves the same $\Pi_{3}^{1}$-sentences as $\FIVE$ by \cite{yamayamaharehare}*{Theorem 2.2}.   
By definition, the Suslin functional $S^{2}$ can decide whether a $\Sigma_{1}^{1}$-formula (as in the left-hand side of $(S^{2})$) is true or false.   
Note that we allow formulas with (type one) \emph{function} parameters, but \textbf{not} with (higher type) \emph{functional} parameters.
Similarly, let $(S_{k}^{2})$ state the existence of a functional $S_{k}^{2}$ that decides $\Pi_{k}^{1}$-formulas (only involving type zero and one parameters).  
We define $\SIXK\equiv \RCAo+(S_{k}^{2})$.

\smallskip
\noindent
Thirdly, full second-order arithmetic $\Z_{2}$ is readily derived from $\cup_{k}\SIXK$, or from:
\be\tag{$\exists^{3}$}
(\exists E^{3}\leq_{3}1)(\forall Y^{2})\big[  (\exists f^{1})Y(f)=0\asa E(Y)=0  \big], 
\ee
and we define $\Z_{2}^{\Omega}\equiv \RCAo+(\exists^{3})$ and $\Z_{2}^{\omega}\equiv \cup_{k}\SIXK$; both are conservative over $\Z_{2}$ by \cite{hunterphd}*{Cor.~2.6}, but see Remark \ref{knellen}.   
The (unique) functional from $(\exists^{3})$ is also called `$\exists^{3}$', and we will use a similar convention for other functionals.  

\smallskip
\noindent
Fourth, the \emph{comprehension for Cantor space} functional, introduced in \cite{dagsamV}, yields a conservative extension of $\WKL_{0}$ by \cite{kohlenbach2}*{Prop.\ 3.15}:
\be\tag{$\kappa_{0}^{3}$}
(\exists \kappa_{0}^{3}\leq_{3}1)(\forall Y^{2})\big[ \kappa_{0}(Y)=0\asa (\exists f\in C)(Y(f)>0)  \big].
\ee
Kohlenbach has shown $[(\exists^{2})+(\kappa_{0}^{3})]\asa (\exists^{3})$ over $\RCAo$ by \cite{dagsam}*{Rem.\ 6.13}.  

\smallskip

\noindent
Fifth, recall that the Heine-Borel theorem (aka \emph{Cousin's lemma}; see \cite{cousin1}*{p.\ 22}) states the existence of a finite sub-cover for an open cover of a compact space. 
Now, a functional $\Psi:\R\di \R^{+}$ gives rise to the \emph{canonical} cover $\cup_{x\in I} I_{x}^{\Psi}$ for $I\equiv [0,1]$, where $I_{x}^{\Psi}$ is the open interval $(x-\Psi(x), x+\Psi(x))$.  
Hence, the uncountable cover $\cup_{x\in I} I_{x}^{\Psi}$ has a finite sub-cover by the Heine-Borel theorem; in symbols:
\be\tag{$\HBU$}
(\forall \Psi:\R\di \R^{+})(\exists  y_{1}, \dots, y_{k}\in I){(\forall x\in I)}(\exists i\leq k)(x\in I_{y_{i}}^{\Psi}).
\ee
There is also the highly similar \emph{Lindel\"of lemma} stating the existence of a \emph{countable} sub-cover of possibly non-compact spaces.  We restrict ourselves to $\R$ as follows. 
\be\tag{$\LIN$}
(\forall \Psi:\R\di \R^{+})(\exists \Phi^{0\di 1})(\forall x\in \R)(\exists n^{0})(x\in I^{\Psi}_{\Phi(n)}),
\ee
By the results in \cite{dagsamIII, dagsamV}, $\Z_{2}^{\Omega}$ proves $\HBU$, but $\SIXK$ cannot (for $k\geq 1$). The same holds for $\LIN$, if we add $\QFAC^{0,1}$, while the latter implies $\HBU\asa [\WKL+\LIN]$.
The importance/naturalness of $\HBU$ and $\LIN$ is discussed in Section \ref{introke}. 

\smallskip

Finally, since Cantor space (denoted $C$ or $2^{\N}$) is homeomorphic to a closed subset of $[0,1]$, the former inherits the same property.  
In particular, for any $G^{2}$, the corresponding `canonical cover' of $2^{\N}$ is $\cup_{f\in 2^{\N}}[\overline{f}G(f)]$ where $[\sigma^{0^{*}}]$ is the set of all binary extensions of $\sigma$.  By compactness, there is a finite sequence $\langle f_0 , \ldots , f_n\rangle$ such that the set of $\cup_{i\leq n}[\bar f_{i} F(f_i)]$ still covers $2^{\N}$.  By \cite{dagsamIII}*{Theorem 3.3}, $\HBU$ is equivalent to the same compactness property for $C$, as follows:
\be\tag{$\HBU_{\c}$}
(\forall G^{2})(\exists  f_{1}, \dots, f_{k} \in C ){(\forall f\in C)}(\exists i\leq k)(f\in [\overline{f_{i}}G(f_{i})]).
\ee
We now introduce the specification $\SCF(\Theta)$ for a functional $\Theta^{2\di 1^{*}}$ which computes such a finite sub-cover.  
We refer to such a functional $\Theta$ as a \emph{realiser} for the compactness of Cantor space, and simplify its type to `$3$' to improve readability.
\be\label{normaal}\tag{$\SCF(\Theta)$}
(\forall G^{2})(\forall f \in C)(\exists g\in \Theta(G))(f\in [\overline{g}G(g)]).
\ee
Clearly, there is no unique $\Theta$ as in \ref{normaal} (just add more binary sequences to $\Theta(G)$); nonetheless, 
we have in the past referred to any $\Theta$ satisfying $\SCF(\Theta)$ as `the' \emph{special fan functional} $\Theta$, and we will continue this abuse of language.  


\section{Reverse Mathematics of Topology}
We study the RM of theorems of topology pertaining to the following notions: (topological) dimension and the Urysohn identity (Section \ref{diemensie}) and paracompactness (Section \ref{diemensie2}). 
We introduce a suitable notion of cover (Section \ref{introke}) and show (Section \ref{kerkend}) that our aforementioned results are independent of the definition of cover, without making use of the axiom of choice. 
We discuss similar results for the Lindel\"of lemma and partitions of unity (Section \ref{kerkend}).  
We formulate a most elegant base theory in Section \ref{BBB} based on the \emph{neighbourhood function principle} from \cite{troeleke1}.    
\subsection{Introduction: topology in higher-order arithmetic}\label{introke}
We discuss how higher-order arithmetic can accommodate the central topological notion of cover.  
In particular, we introduce a generalisation of the notion of cover used in \cite{dagsamIII, dagsamV} and shall show in Section \ref{kerkend} that the new notion yields covering lemmas
equivalent to the original, \emph{without} a need for the axiom of countable choice. 

\smallskip

First of all, early covering lemmas, like the \emph{Cousin and Lindel\"of lemmas}, did not make use of the (general) notion of cover. 
Indeed, Cousin and Lindel\"of talk about (uncountable) covers of $\R^{2}$ and $\R^{n}$ as follows (resp.\ in 1895 and 1903): 
\begin{quote}
we suppose that to each point of $S$ corresponds a circle of non-zero finite radius and with this point as centre (\cite{cousin1}*{p.\ 22} )
\end{quote}
\begin{quote}
for every point $\textsf{\textup{P}}$, let us construct a sphere $\textsf{\textup{S}}_{\textsf{\textup{P}}}$ with $\textsf{P}$ as the centre\\ and a variable radius $\rho_{\textsf{\textup{P}}}$ (\cite{blindeloef}*{p.\ 698}))
\end{quote}
To stay close to the original formulation by Cousin and Lindel\"of, we introduced in \cite{dagsamIII, dagsamV} the notion of `canonical' open covers $\cup_{x\in I}I_{x}^{\Psi}$ of $I\equiv [0,1]$ generated by $\Psi:I\di \R^{+}$ and where $I_{x}^{\Psi}\equiv (x-\Psi(x), x+\Psi(x))$.
Unfortunately, such covers always involve points that are covered by arbitrarily many intervals; this property makes such covers unsuitable for e.g.\ the study of topological dimension, in which the (minimal) number of intervals covering a point is central.  

\smallskip

Secondly, the previous observation shows that we have to generalise our notion of canonical cover, and we shall do this by considering $\psi:I\di \R$. i.e.\ we allow empty $I_{x}^{\psi}$. 
In this way, we say that `$\cup_{x\in I}I_{x}^{\psi}$ covers $[0,1]$' if $(\forall x\in I)(\exists y\in I)(x\in I_{y}^{\psi})$.  
This notion of cover gives rise to the following version of the Heine-Borel theorem. 
\be\tag{$\HBT$}
(\forall \psi:I\di \R)\big[ I\subset \cup_{x\in I}I_{x}^{\psi}\di   (\exists y_{1}, \dots, y_{k} \in I)(I\subset \cup_{i\leq k}I_{y_{i}}^{\psi}) \big].
\ee
We establish in Section \ref{kerkend} that our `new' notion of cover is quite robust by showing that (i) $\HBU\asa \HBT$ over $\RCAo+\QFAC^{1,1}$, i.e.\ the new notion of cover is not a real departure from the old one, and (ii) the previous equivalence can also be proved without the axiom of choice.  
Item (ii) should be viewed in the light of the topological `disasters' (see e.g.\ \cite{kermend}) that apparently happen in the absence of the axiom of (countable) choice.  We also show that any notion of cover definable in $\Z_{2}^{\Omega}$ inherits the aforementioned `nice' properties. 
Thus, we may conclude that our results boast a lot of robustness, and in particular that they do not depend on the definition of cover, even in the absence of the axiom of (countable) choice. 
 
\smallskip

Finally, we discuss the mathematical naturalness of $\HBU$ and $(\exists^{2})$.
\begin{rem}\rm
Dirichlet already discusses the  characteristic function of the rationals, which is essentially $\exists^{2}$, around 1829 in \cite{didi1}, while Riemann defines a function with countably many discontinuities via a series in his \emph{Habilitationsschrift} (\cite{kleine}*{p.~115}).  
Furthermore, the \emph{Cousin lemma} from \cite{cousin1}*{p.\ 22}, which is essentially $\HBU$, dates back\footnote{The collected works of Pincherle contain a footnote by the editors (see \cite{tepelpinch}*{p.\ 67}) which states that the associated \emph{Teorema} (published in 1882) corresponds to the Heine-Borel theorem.  Moreover, Weierstrass proves the Heine-Borel theorem (without explicitly formulating it) in 1880 in \cite{amaimennewekker}*{p.\ 204}.   A detailed motivation for these claims may be found in \cite{medvet}*{p. 96-97}.} about 135 years.  
As shown in \cite{dagsamIII}, $(\exists^{2})$ and $\HBU$ are essential for the development of the \emph{gauge integral} (\cite{bartle1337}).  This integral was introduced by Denjoy (\cite{ohjoy}), in a different and more complicated form, around the same time as the Lebesgue integral; the reformulation of Denjoy's integral by Henstock and Kurzweil in Riemann-esque terms (see \cite{bartle1337}*{p.\ 15}), provides a \emph{direct} and elegant formalisation of the \emph{Feynman path integral} (\cites{burkdegardener,mullingitover,secondmulling}) and financial mathematics (\cites{mulkerror, secondmulling}).    
\end{rem}

\subsection{The notion of dimension}\label{diemensie}
The notion of dimension of basic spaces like $[0,1]$ or $\R^{n}$ is intuitively clear to most mathematicians, but finding a formal definition of dimension \emph{that does not depend on the topology} is a non-trivial problem.  

\smallskip

We introduce three notions of dimension: the  `topological' dimension $\dim X$ and the `small' and `large' inductive dimensions $\textup{ind } X$ and $\textup{Ind } X$.  We study the RM properties of the \emph{Uryoshn identity} (\cite{enc2}*{p.\ 272}) which expresses that these dimension are equal for a large class of spaces, including separable metric spaces.

\smallskip

First of all, the \emph{covering dimension}, later generalised to the \emph{topological dimension}, goes back to Lebesgue.  
Indeed, Munkres writes the following:
\begin{quote}
We shall define, for an arbitrary topological space $X$, a notion of topological dimension. It is the ``covering dimension" originally defined by Lebesgue.  (\cite{munkies}*{p.\ 305})
\end{quote}
The following definition of topological dimension may be found in Munkres' seminal monograph \cite{munkies}*{p.\ 161}, and in \cite{enc2}*{p.\ 274}, \cite{engeltjemijn}*{Ex.\ 1.7.E and Prop.\ 3.2.2}.  
\bdefi[Order]
A collection $\mathcal{A}$ of subsets of the space $X$ is said to have order $m + 1$, 
if some point of $A$ lies in $m +1$ elements of $\mathcal{A}$, and no point of $X$ lies in more than $m +1$
elements of $A$.
\edefi
\bdefi[Refinement]
Given a collection $\mathcal{A}$ of subsets of $X$, a collection $\mathcal{B}$ is said to refine $\mathcal{A}$, or to be
a refinement of $\mathcal{A}$ if for each element $B \in \mathcal{B}$ there is an element $A\in \mathcal{A}$ such that $A\subset B$.
\edefi
\bdefi[Topological dimension]
A space $X$ is said to be \textbf{finite-dimensional} if there is $m\in\N$ such
that for every open covering $\mathcal{A}$ of $X$, there is an open covering $\mathcal{B}$ of $X$ that refines $\mathcal{A}$
and has order at most $m + 1$. The topological dimension of $X$ is the
smallest value of $m$ for which this statement holds; we denote it by $\dim X$.
\edefi
In the context of $\RCAo$, we say that `$\phi:I\di \R$ is a \emph{refinement} of $\psi:I\di \R$' if $(\forall x\in I)(\exists y\in I)(I_{x}^{\phi}\subseteq I_{y}^{\psi})$.  
With this definition in place, statements like `the topological dimension of $[0,1]$ is at most $1$', denoted `$\dim([0,1])\leq1$', makes perfect sense in $\RCAo$.  
Such a statement turns out to be quite hard to prove, as full second-order arithmetic is needed to prove $\HBT$ by Theorem \ref{ziedenauw}.  
\begin{thm}\label{rathergen}
The system $\ACAo+\QFAC^{1,1}+[\dim([0,1])\leq1]$ proves $\HBT$. 
\end{thm}
\begin{proof}
Let $\psi:I\di \R$ be such that $\cup_{x\in I}I_{x}^{\psi}$ covers $[0, 1]$, and let $\phi:I\di \R$ be the associated refinement of order at most $1$.  
Since the innermost formula is $\Sigma_{1}^{0}$ (with parameters), we may apply $\QFAC^{1,1}$ to $(\forall x\in I)(\exists y\in I)(x\in I_{y}^{\phi})$ to obtain $\Xi^{1\di 1}$ such that $\Xi(x)$ provides such $y$.
Define $\zeta^{0\di 1}$ as follows: $\zeta(0):=\Xi(0)+ \phi(\Xi(0))$ and $\zeta(n+1):= \Xi(\zeta(n))+\phi(\Xi(\zeta(n)))$.   Now consider the following formula:
\be\label{tuigs}
(\exists x\in I )(\forall n\in \N)(\zeta(n)<_{\R}x).
\ee
If \eqref{tuigs} is false, take $x=1$ and note that if $\zeta(n_{0})\geq_{\R}1$, the finite sequence $I_{\Xi(0)}^{\phi}, I_{\Xi(\zeta(0))}^{\phi}, I_{\Xi(\zeta(1))}^{\phi}, \dots, I_{\Xi(\zeta(n_{0}+1))}^{\phi}$ yield a finite sub-cover of $\cup_{x\in I}I_{x}^{\phi}$.  In this case, we apply $\QFAC^{1,1}$ (using also $(\exists^{2})$) to $(\forall x\in I)(\exists y\in I)(I_{x}^{\phi}\subseteq I_{y}^{\psi})$ 
to go from a finite sub-cover of $\cup_{x\in I}I_{x}^{\phi}$ to a finite sub-cover of $\cup_{x\in I}I_{x}^{\psi}$, and $\HBT$ follows.  

\medskip

If \eqref{tuigs} is true, let $x_{0}\in I$ be the least $x\in I$ such that $\varphi(x)\equiv(\forall n\in \N)(\zeta(n)<_{\R}x)$.  Since $\varphi(x)$ is $\Pi_{1}^{0}$, we can use $\exists^{2}$ and the usual interval-halving technique to find $x_{0}$; alternatively, use the monotone convergence theorem (\cite{simpson2}*{III.2.2}), provable in $\ACA_{0}$.  
However, $I_{\Xi(x_{0})}^{\phi}$ covers $x_{0}$, and thus for $n_{1}$ large enough, $\zeta(n)$ for $n\geq n_{1}$ will all be in the former interval, by the leastness of $x_{0}$.  But then there are points of order $3$ in the (by definition non-empty) intersection of $I_{\Xi(\zeta(n_{1}))}^{\phi}$ and $I_{\Xi(\zeta(n_{1}+1))}^{\phi}$, as this intersection is also inside $I_{\Xi(x_{0})}^{\phi}$.  This observation contradicts the assumption $\dim([0,1])\leq1$, and hence \eqref{tuigs} must be false, and we are done.  
\end{proof}
The previous theorem has a number of corollaries.  First of all, we obtain an equivalence over a weak base theory; we believe the components of the left-hand side to be independent\footnote{Firstly,  $\Z_{2}^{\Omega}+\QFAC^{0,1}$ does not prove $\HBU$ (\cite{dagsamIII, dagsamV}).  Secondly, $\dim(I)= 1$ seems consistent with recursive mathematics by \cite{beeson1}*{Theorem 6.1, p.\ 69}, i.e.\ the former cannot imply $\WKL$.}, i.e.\ that a proper `splitting' of $\HBT$ is achieved.

\begin{cor}\label{dorkeeee}
$\RCAo+\QFAC^{1,1}$ proves that $\big(\WKL+ [\dim(I)=1]\big)\asa \HBT$.
\end{cor}
\begin{proof}
For the forward direction, in case $(\exists^{2})$, the proof of the theorem goes through.  In case $\neg(\exists^{2})$, all $F:\R\di \R$ are continuous, while all $F^{2}$ are continuous on Baire space, and hence uniformly continuous (and thus bounded) on Cantor space by $\WKL$ (see \cite{kohlenbach2}*{Prop.\ 3.7 and 3.12} and \cite{kohlenbach4}*{Prop.~4.10}).  Now consider the following statement, which (only) holds since $\psi:I\di \R$ is continuous:
\be\label{piew}\textstyle
(\forall f\in C)(\exists q\in I\cap \Q)(\exists n\in \N)\underline{(\r(f)\in I_{q}^{\psi}\wedge \psi(q)\geq \frac{1}{2^{n}})}, 
\ee
where $\r(f)$ is $\sum_{n=0}^{\infty}\frac{f(n)}{2^{n}}$ for binary $f$, and where the underlined formula is $\Sigma^{0}_{1}$.  
Applying $\QFAC^{1,0}$ to \eqref{piew}, there is $\Xi^{2}$ such that $n\leq \Xi(f)$ in \eqref{piew}.  Since $\Xi$ is bounded on $C$, there is $N_{0}\in \N$ such that
\be\label{pieq2}\textstyle
(\forall f\in C)(\exists q\in I\cap \Q){(\r(f)\in I_{q}^{\psi}\wedge \psi(q)\geq \frac{1}{2^{N_{0}}})}, 
\ee
which immediately implies that $\cup_{x\in I}I_{x}^{\psi}$ has a finite sub-cover (generated by rationals), and the latter may be found by applying $\QFAC^{1,0}$ to \eqref{pieq2} and iterating the choice function at most $2^{N_{0}+1}$ times. 
Since $\frac{i}{2^{n}}$ has an obvious binary representation, we do not need to convert arbitrary $x\in I$ to binary.  
We obtain $\HBT$ in each case, and $(\exists^{2})\vee \neg(\exists^{2})$ finishes this direction of the proof. 

\smallskip

For the reverse direction, note that $\HBT\di \HBU\di \WKL$.  To prove $\dim(I)=1$, the finite sub-cover provided by $\HBT$ is readily converted to a refinement of order $1$ using $\exists^{2}$, as the latter functional can decide equality between real numbers.  
Now, in case $\neg(\exists^{2})$, obtain \eqref{pieq2} in the same way as above, and let $\Xi$ be a choice function that provides $\Xi(f)=q$.  
Define $\zeta$ as follows: $\zeta(0):=\Xi(00\dots)+\frac{1}{2^{N_{0}}} $ and $\zeta(n+1):= \Xi(\zeta(n))+\frac{1}{2^{N_{0}}}$.  
For $n> 2^{N_{0}+1}$, this function readily yields a finite open cover of $I$ that is also a refinement of the cover generated by $\psi$. 
Since all points are rationals, we can refine this cover to have order $1$, and $(\exists^{2})\vee \neg(\exists^{2})$ finishes the proof. 
\end{proof}
For future reference, we note that the proof also establishes $\RCAo+\neg(\exists^{2})+\WKL\vdash \HBT$, i.e.\ the axiom of choice is not used. 

\smallskip

It is a natural question (posed before by Hirschfeldt; see \cite{montahue}*{\S6.1}) whether the axiom of choice is really necessary in the previous (and below) theorems.  
We answer this question in the negative in Section \ref{kerkend}.

\smallskip

%
%
Next, in order to prove the next corollary concerning Urysohn's identity, we introduce the notion of inductive definition as in \cite{engeltjemijn}*{\S1.1.1}.
\bdefi[Inductive dimension]\label{roolin} We inductively define the \emph{small inductive dimension} $ \ind X$ for a topological space $X$ as follows. 
\begin{enumerate}
\item[(d1)] For the empty set $\emptyset$, we define ${\textup{ind }} \emptyset={\textup{Ind }}\emptyset=-1$;
\item[(d2)] $\ind X\leq n$, where $n = 0,1,\dots,$ if for every point $x \in X$ and each neighbourhood $V \subset X$ of the point $x$ there exists an open set $U \subset X$ such that $x\in U\subset V$ and $\ind (\partial U)<n-1$;
\item[(d3)]  $\ind X = n$ if $\ind X \leq n$ and $\ind X > n- 1$, i.e., the inequality $\ind X <n- 1$ does not hold;
\item[(d4)] $\ind X= \infty$ if $\ind X>n$ for $n=-1,0,1,...$.  
\end{enumerate}
The \emph{large inductive dimension} $\Ind X$ is obtained by replacing (d2) by:
\begin{enumerate}
\item[(d$2^{*}$)] $\Ind X<n$, where $n=0,1,...,$ if for every closed set $A\subset X$ and
each open set $V \subset X$ which contains the set $A$ there exists an open set $U \subset X$ such that $A\subset U \subset V$ and $\Ind (\partial U)< n-1$.
\end{enumerate}
If $X$ is Euclidean space, $V$ is generally chosen to be a ball centred at $x$.
\edefi 
In light of Definition \ref{roolin}, the (small and large) inductive dimension of singletons of real numbers, or the unit interval, makes sense in $\RCAo$, and is respectively $0$ and $1$.  
Moreover, the \emph{Urysohn identity} is the statement that $\dim X=\textup{ind} X=\textup{Ind} X$, and holds for a large class of spaces $X$; this identity constitutes one of the main problems in \emph{dimension theory}, according to \cite{enc2}*{p.\ 274}, while it is called the \emph{the fundamental theorem of dimension theory} in \cite{engeltjemijn}.
\begin{cor}\label{hoerke}
The system $\RCAo+\QFAC^{1,1}$ proves that $\HBT$ is equivalent to: the conjunction of  $\WKL$ and Urysohn's identity for the unit interval.
\end{cor}
\begin{proof}
Immediate from Corollary \ref{dorkeeee}.
\end{proof}

\subsection{Paracompactness}\label{diemensie2}
The notion of \emph{paracompactness} was introduced in 1944 by Dieudonn\'e in \cite{nogeengodsgeschenk} and plays an important role in the characterisation of metrisable spaces via e.g.\ \emph{Smirnov's metrisation theorem} (\cite{munkies}*{p.\ 261}).  
The fact that every metric space is paracompact is \emph{Stone's theorem} (see \cites{goodgoing,stoner2} and \cite{munkies}*{p.\ 252}). 

\smallskip

Our interest in paracompactness stems in part from its occurrence in classical RM (see e.g.\ \cites{simpson2, mummymf, mummyphd}), as detailed in Remark \ref{diemummy}.  
The aim of this section is to show that there is a \emph{huge} difference in logical and computational hardness between the `second-order/countable' version of paracompactness, and the `actual' definition.  
Indeed, the fact that the unit interval is paracompact implies $\HBT$; moreover, the latter can be `split' into the former plus $\WKL$ by Corollary \ref{XYW}.  

\smallskip

Munkres states the following definition of paracompactness in \cite{munkies}*{p.\ 253}.  
\bdefi[Locally finite]
A collection $\mathcal{A}$ of subsets of a space $X$ is \emph{locally finite} if any $x\in X$ has a neighbourhood that intersects only finitely many $A\in \mathcal{A}$.
\edefi
\bdefi[Paracompact]
A space $X$ is \emph{paracompact} if every open covering $\mathcal{A}$ of $X$ has a locally
finite open refinement $\mathcal{B}$ that covers $X$.
\edefi
With these definitions, the statement that the unit interval is paracompact, makes sense in $\RCAo$.  
By \emph{Stone's theorem}, a metric space is paracompact, but this fact is not provable in $\ZF$ alone (see \cite{goodgoing}).  
Similarly, Stone's theorem \emph{for the unit interval} is not provable in any system $\SIXK$ by the following theorem.  
Note that the results in Section \ref{kerkend} yield a proof in $\Z_{2}^{\Omega}$ of the paracompactness of $[0,1]$.  
\begin{thm}\label{paramaeremki}
The system $\ACAo+\QFAC^{1,1}+ \textup{`$[0,1]$ is paracompact'}$ proves $\HBT$. 
\end{thm}
\begin{proof}
We use the proof of Theorem \ref{rathergen} with minor modification.
Let $\psi:I\di \R$ be such that $\cup_{x\in I}I_{x}^{\psi}$ covers $[0, 1]$, and let $\phi:I\di \R$ be a locally finite refinement. 
Assume \eqref{tuigs}, where let $x_{0}\in I,\zeta^{0\di 1}, \Xi^{1,1}$ are as in the aforementioned proof.    
Clearly, any neighbourhood of $x_{0}$ will contain all intervals $I^{\phi}_{\Xi(\zeta(n))}$ for $n$ large enough.  
This observation contradicts the assumption that $[0,1]$ is paracompact, and hence \eqref{tuigs} must be false, implying $\HBT$ as in the proof of Theorem \ref{rathergen}.  
\end{proof}
The following corollary is proved in the same way as Corollary \ref{dorkeeee}; the left-hand side constitutes a proper `splitting' of $\HBT$, as the $\ECF$-translation of `$[0,1]$ is paracompact' is essentially the statement that $[0,1]$ is \emph{countably} paracompact, and the latter is provable in $\RCA_{0}$ by \cite{simpson2}*{II.7.2}.
\begin{cor}\label{XYW}
 $\RCAo+\QFAC^{1,1}$ proves $[\WKL + \textup{`$[0,1]$ is paracompact'}]\asa \HBT$. 
\end{cor}
Another interpretation of the previous corollary is as follows: by the results in \cite{wienszoon}, the notion of compactness is equivalent to `paracompact plus pseudo-compact' for a large class of spaces, and pseudo-compactness essentially expresses that continuous functions are bounded on the space at hand, i.e.\ the pseudo-compactness of $[0,1]$ is equivalent to $\WKL$ by \cite{simpson2}*{IV.2.3} and \cite{kohlenbach4}*{Prop.\ 4.10}. 

\smallskip

The following remark highlights the difference between `actual' and `second-order/countable' paracompactness.  It also suggests formulating Corollary \ref{thetam}. 
\begin{rem}[Paracompactness in second-order RM]\label{diemummy}\rm
Simpson proves in \cite{simpson2}*{II.7.2} that over $\RCA_{0}$, complete separable metric spaces are \emph{countably} paracompact\footnote{The notion of `countably paracompact' is well-known from Dowker's theorem (see e.g.\ \cite{ooskelly}*{p.~172}), but Simpson and Mummert do not use the qualifier `countable' in \cites{simpson2,mummymf}.}, 
and Mummert in \cite{mummymf}*{Lemma 4.11} defines a realiser for paracompactness as in \cite{simpson2}*{II.7.2} inside $\ACA_{0}$.  This realiser plays a crucial role in the proof of Mummert's metrisation theorem, called `MFMT', inside $\SIX$ (see \cite{mummymf}*{\S4}).  
Note that $\SIX$ occurs elsewhere in the RM of topology (\cite{mummy, mummyphd}).  By Theorem~\ref{paramaeremki}, the (higher-order) statement \emph{the unit interval is paracompact} is equivalent to $\HBT$, and hence not provable in $\cup_{k}\SIXK$, i.e.\ there is a 
\emph{huge} difference in strength between `second-order/countable' and `actual' paracompactness.  In fact, the logical hardness of the aforementioned statement dwarfs $\SIX$ from the RM of topology.  
\end{rem}
Let us call $\Omega^{\mathbb{1}\di\mathbb{1}}$ a `realiser for the paracompactness of $[0,1]$' if $\Omega(\psi)(1):I\di \R$ yields a locally finite open refinement of the cover associated to $\psi:I\di \R$, and if 
\be\label{popol}
(\forall x\in I)(I_{x}^{\Omega(\psi)(1)}\subseteq I_{\Omega(\psi)(2)(x)}^{\psi}),
\ee
i.e.\ the refining cover is `effectively' included in the original one, just like in \cite{simpson2, mummymf}.
\begin{cor}\label{thetam}
A realiser $\Omega^{\mathbb{1}\di \mathbb{1}}$ for the paracompactness of $[0,1]$, together with Feferman's $\mu$, computes $\Theta$ such that $\SCF(\Theta)$ via a term of G\"odel's $T$.
\end{cor}
\begin{proof}
Immediate from the proof of Theorems \ref{rathergen} and \ref{paramaeremki}.  Note that $\Xi$ is the identity function in case we consider covers generated by $\Psi:I\di \R^{+}$ as in $\HBU$. 
Furthermore, a realiser for $\HBU$ computes a realiser for $\HBU_{\c}$, i.e.\ the special fan functional, via a term in G\"odel's $T$, as discussed in \cite{dagsamIII}*{\S3.1}
\end{proof}
As it turns out, the condition \eqref{popol} for a realiser for paracompactness has already been considered, namely as follows.  
\begin{quote}
all proofs of Stone's Theorem (known to the authors) actually prove a stronger conclusion which implies $\mathsf{AC}$. It
is based on an idea from [\dots]. Let us call a refinement $\mathcal{V}$ of $\mathcal{U}$ \emph{effective} if there is
a function $a : \mathcal{V} \di \mathcal{U}$ such that $V \subset a(V)$ for all $V \in  \mathcal{V}$. (\cite{goodgoing}*{p.\ 1217})
\end{quote}
As it turns out, the notion of `effectively paracompact' is intimately connected to the Lindel\"of lemma, as discussed in Section \ref{unite}. 


\subsection{Covers in higher-order arithmetic}\label{kerkend}
In Section \ref{introke}, we introduced a generalisation of the notion of cover used in \cite{dagsamIII, dagsamV}, while we used this notion in Sections~\ref{diemensie} and \ref{diemensie2} to obtain RM results.  
In this section, we show that these RM results have some robustness: we show that the new notion of cover yields covering lemmas equivalent to the original ones (with the definition from \cite{dagsamIII, dagsamV}), even in the absence of the axiom of choice. 
We also show that \emph{any} notion of cover definable in second-order arithmetic inherits these `nice' properties.  We treat the Heine-Borel theorem, the Lindel\"of lemma, as well as theorems pertaining to partitions of unity. 
\subsubsection{The Heine-Borel theorem}
We prove $\HBT\asa \HBU$ with and without the axiom of choice in the base theory. 
In this way, we observe that our new notion of cover does not really change the Heine-Borel theorem.
\begin{thm}\label{ziedenauw}
The system $\RCA_{0}^{\omega}+\QFAC^{1,1}$ proves $\HBU\asa \HBT$.
\end{thm}
\begin{proof}
The reverse direction is immediate.  For the forward direction, in case $\neg(\exists^{2})$, we obtain $\HBU\di \WKL$ and proceed as in the proof of Corollary \ref{dorkeeee}. 
In case $(\exists^{2})$, let $\psi$ be as in $\HBT$ and consider $(\forall x\in I)(\exists y\in I)(x\in I_{y}^{\psi})$.  Since the innermost formula is $\Sigma_{1}^{0}$, we may 
apply $\QFAC^{1,1}$ to obtain $\Xi$ such that $(\forall x\in I)(x\in I_{\Xi(x)}^{\psi})$.  Since $\exists^{2}$ provides a functional that converts real numbers in $I$ to a unique binary representation, we may assume that $\Xi$ is extensional on the reals.   
Now define $\Psi:I\di \R^{+}$ by $\Psi(x):=\min\big(|x-(\Xi(x)-\psi(\Xi(x)))|, |x-(\Xi(x)+\psi(\Xi(x)))|\big)$, and note that $I^{\Psi}_{x}\subseteq I_{\Xi(x)}^{\psi}$.  
Applying $\HBU$, we obtain a finite sub-cover of $\cup_{x\in I}I_{x}^{\Psi}$, say generated by $y_{1}, \dots, y_{k}\in I $, and $\cup_{i\leq k}I_{\Xi(y_{i})}^{\psi}$ is then a finite sub-cover of $\cup_{x\in I}I_{x}^{\psi}$.
\end{proof}
Recall that $\HBU$ is provable in $\Z_{2}^{\Omega}$ by \cite{dagsamV}*{\S4}, i.e.\ without the axiom of choice.  
While the use of $\QFAC^{1,1}$ in $\HBU\di \HBT$ \emph{seems} essential, it is in fact not, by the following theorem.  
Note that $\textsf{IND}$ is the induction axiom for all formulas in the language of $\RCAo$; the base theory is not stronger than Peano arithmetic. 
\begin{thm}\label{fugu}
The system $\RCAo+\textsf{\textup{IND}}+(\kappa_{0}^{3})$ proves $\HBU\asa \HBT$
\end{thm}
\begin{proof}
The reverse direction is immediate.  For the forward direction, in case $\neg(\exists^{2})$, we obtain $\HBU\di \WKL$ and proceed as in the proof of Corollary \ref{dorkeeee}. 
In case of $(\exists^{2})$, let $\psi$ be as in $\HBT$ and note that $(\forall x\in I)(\exists y\in I)(x\in I_{y}^{\psi})$ implies:
\be\label{centrifuge}\textstyle
(\forall x\in I)(\exists n\in \N)\underline{(\exists y\in I)((x-\frac{1}{2^{n}}, x+\frac{1}{2^{n}})\subseteq I_{y}^{\psi})}, 
\ee
where the underlined formula is decidable thanks to $(\exists^{3})\equiv [(\exists^{2}) + (\kappa_{0}^{3})]$.
Hence, applying $\QFAC^{1,0}$ to \eqref{centrifuge}, we obtain $\Psi:I\di \R^{+}$ such that $\cup_{x\in I}I_{x}^{\Psi}$ is a canonical cover of $I$.  
Applying $\HBU$, we obtain a finite sub-cover of $\cup_{x\in I}I_{x}^{\Psi}$, say generated by $x_{1}, \dots, x_{k}\in I $.  
By definition, we have $(\forall x\in I)(\exists y\in I)(I^{\Psi}_{x}\subseteq I_{y}^{\psi})$, and 
\be\label{conplete}
(\forall w^{1^{*}})(\exists v^{1^{*}})(\forall i<|w|)(I^{\Psi}_{w(i)}\subseteq I_{v(i)}^{\psi})
\ee
follows from $\textsf{IND}$ by induction on $|w|$.  Applying \eqref{conplete} for $w=\langle x_{1}, \dots, x_{k}\rangle$, we obtain a finite sub-cover for $\cup_{x\in I}I_{x}^{\psi}$. The law of excluded middle finishes the proof. 
\end{proof}
As to open questions, we do not know if the base theory proves $\HBT$ outright or not.  
Similarly, we do not know if $\RCAo+(\kappa_{0}^{3})$ proves $\WKL$ or not.  

\smallskip

In conclusion, we mention two important observations that stem from the above.

\smallskip

First of all, it is easy to see that the first two proofs go through for the Heine-Borel theorem for $[0,1]$ based on {any `reasonable' notion} of cover.
Indeed, as long as the formulas `$x\in U_{y}$' and `$[a,b]\subseteq U_{x}$' for the new notion of cover $\cup_{x\in I}U_{x}$ of $I$ are decidable in $\Z_{2}^{\Omega}$, the above proofs go through (assuming $(\kappa_{0}^{3})$).  
Since $\Z_{2}^{\Omega}$ can decide if $Y:\R\di \{0,1\}$ represents an open subset of $\R$ (using the textbook definition of open set), this notion of `reasonable' seems quite reasonable. 

\smallskip

Secondly, emulating the proof of Theorem \ref{fugu}, we observe that the above results go through in the base theory with $(\kappa_{0}^{3})+\textsf{IND}$ instead of $\QFAC^{1,1}$.  
These include Theorem~\ref{rathergen}, Corollary \ref{dorkeeee}, Corollary \ref{hoerke}, Theorem~\ref{paramaeremki}, and Corollary~\ref{XYW}. 
Thus, these results do not require the axiom of choice in the base theory.  

\subsubsection{The Lindel\"of lemma}
We show that the Lindel\"of lemma does not depend on the definition of cover, similar to the case of the Heine-Borel theorem.  
On one hand, since $[\LIN+\WKL]\asa \HBU$, one expects such results.  
On the other hand, as shown in \cite{dagsamV}*{\S5}, the strength of the Lindel\"of lemma is highly dependent on the exact\footnote{The countable sub-cover in the Lindel\"of lemma can be given by a sequence of reals generating the intervals (strong version), or just a sequence of intervals (weak version). 
The strong version implies $\QFAC^{0,1}$ and hence is unprovable in $\ZF$, while the weak version is provable in $\Z_{2}^{\Omega}$.  
} formulation, but this dependence is not problematic for our context.  
  
\smallskip  
  
We introduce the notion of cover used in \cite{dagsamV}*{\S5}, as follows.
We consider $\psi:I\di \R^{2}$ and covers $\cup_{x\in I}J_{x}^{\psi}$ in which the interval $J_{x}^{\psi}:=(\psi(x)(1), \psi(x)(2))$ is potentially empty but $(\forall x\in I)(\exists y\in I)(x\in J_{y}^{\psi})$. 
This notion of cover yields a `strong' version of the Lindel\"of lemma, as follows.  
\be\tag{$\LIL$}
(\forall \psi:\R\di \R^{2})\big[ \R\subseteq \cup_{x\in \R}J_{x}^{\psi}\di   (\exists f:\N\di \R)(\R\subseteq \cup_{n\in \N}J_{f(n)}^{\psi}) \big].
\ee
Similar to the proof of \cite{dagsamIII}*{Theorem 3.13}, one proves that $\HBT\asa [\WKL+\LIL]$ over $\RCAo+\QFAC^{1,1}$.
We first prove that the Lindel\"of lemma $\LIL$ is equivalent to $\LIN$ from \cite{dagsamIII}*{\S3}. 
We believe that $\LIN$ does not imply countable choice $\QFAC^{0,1}$. 
\begin{thm}\label{brokken}
The system $\RCA_{0}^{\omega}+\QFAC^{1,1}$ proves $\LIN\asa \LIL$.
\end{thm}
\begin{proof}
Similar to the proof of Theorem \ref{ziedenauw}: the reverse direction is immediate, while in case of $\neg(\exists^{2})$ each principle is provable in $\RCAo$ using the sub-cover consisting of all rationals.
In case of $(\exists^{2})$, let $\psi$ be as in $\LIL$ and consider $(\forall x\in \R)(\exists y\in \R)(x\in J_{y}^{\psi})$.  Since the innermost formula is $\Sigma_{1}^{0}$, we may 
apply $\QFAC^{1,1}$ to obtain $\Xi$ such that $(\forall x\in \R)(x\in J_{\Xi(x)}^{\psi})$.  Since $\exists^{2}$ provides a functional that converts real numbers to a binary representation, we may assume that $\Xi$ is extensional on the reals.   
Now define $\Psi:I\di \R^{+}$ by $\Psi(x):=\min\big(|x-\psi(\Xi(x))(1)|, |x-\psi(\Xi(x))(2)|\big)$, and note that $I^{\Psi}_{x}\subseteq J_{\Xi(x)}^{\psi}$.  
Applying $\LIN$, we obtain a countable sub-cover of $\cup_{x\in I}I_{x}^{\Psi}$, say generated by $\Phi^{0\di1} $, and $\cup_{i\in \N}I_{\Xi(\Phi(i))}^{\psi}$ is a countable sub-cover of $\cup_{x\in I}I_{x}^{\psi}$.
\end{proof}
For completeness, we also mention the following corollary.  
\begin{cor}
$\RCAo+\QFAC^{1,1}$ proves $ \LIL\asa [\dim(\R)\leq 1]\asa \textup{$\R$ is paracompact}$.  
\end{cor}
\begin{proof}
We only prove the equivalence between $\LIL$ and the paracompactness of $\R$. 
By Theorem \ref{brokken}, it suffices to prove $\LIN$.  
In case $\neg(\exists^{2})$, the latter is provable outright, as all $\R\di \R$-functions are continuous, and then the rationals provide a countable sub-cover for any open cover as in $\LIN$.
Similarly, paracompactness reduces to countable paracompactness, and the latter is provable in $\RCAo$ by \cite{simpson2}*{II.7.2}.
In case of $(\exists^{2})$, the paracompactness of $\R$ (and hence $I$ with minor modification) implies $\HBT$ by Theorem \ref{paramaeremki}, and the aforementioned result $\HBT\asa [\WKL+\LIL]$ over $\RCAo+\QFAC^{1,1}$ finishes the forward direction.  The reverse direction is straightforward as $\exists^{2}$ decides inequalities between reals, and hence can easily refine the countable sub-cover provided by $\LIN$.  
\end{proof}
As it turns out, we can avoid the use of $\QFAC^{1,1}$ as follows
\begin{thm}\label{kokken}
The system $\RCA_{0}^{\omega}+(\kappa_{0}^{3})+\textsf{\textup{IND}}$ proves $[\LIN+\QFAC^{0,1}]\asa \LIL$.
\end{thm}
\begin{proof}
For the forward implication, in case $\neg(\exists^{2})$, the rationals provides a countable sub-cover, as all functions on the reals are continuous by \cite{kohlenbach2}*{Prop.\ 3.7}.  In case of $(\exists^{2})$, fix $\psi:\R\di \R^{2}$ as in $\LIL$ and formulate a version of \eqref{centrifuge} as follows:
\be\label{centrifuge22}\textstyle
(\forall x\in \R)(\exists n\in \N)\underline{(\exists y\in \R)\big((x-\frac{1}{2^{n}}, x+\frac{1}{2^{n}})\subseteq J_{y}^{\psi}\big)}, 
\ee
The underlined formula is again decidable thanks to $\exists^{3}$, and $\QFAC^{1,0}$ yields a functional $\Psi:\R\di \R^{+}$ such that the canonical cover $\cup_{x\in \R}I_{x}^{\Psi}$ also covers $\R$. 
Applying $\LIN$, we obtain a functional $\Phi^{0\di 1}$ and the following version of \eqref{conplete}:
\be\label{conplete2}
(\forall n\in \N)(\exists v^{1^{*}})(\forall i\leq n)(I^{\Psi}_{\Phi(i)}\subseteq I_{v(i)}^{\psi}).
\ee
Applying $\QFAC^{0,1}$ to \eqref{conplete2}, we obtain $\LIL$, and this direction is done.  

\smallskip

For the reverse implication, note that $\LIL\di \QFAC^{0,1}$ follows from \cite{dagsamV}*{Theorem~5.3}, \emph{because} the base theory $\RCAo+(\kappa_{0}^{3})$ allows us to generalise the class of covers, as discussed in \cite{dagsamV}*{Remark 5.9}.  With that, we are done. 
\end{proof}
We believe that the previous 
splitting\footnote{In $\LIN$, any $x\in \R$ is covered by $I_{x}^{\Psi}$, while in $\LIL$ any $x\in \R$ is covered by $J_{y}^{\psi}$ \emph{for some $y\in \R$}.  
In the former case, we `know' which interval covers the point, while in the latter case, we only know \emph{that it exists}.    
We believe this (seemingly minor) difference determines whether one can obtain $\QFAC^{0,1}$ (like in the case of $\LIL$) or not (in the case of $\LIN$, we conjecture).  
Indeed, applying $\QFAC^{1,0}$ to the conclusion of $\LIL$, we obtain a functional that provides for any $x\in \R$, an interval $J_{y}^{\psi}$ covering $x$, i.e.\ $\LIL$ clearly exhibits `axiom of choice' behaviour, while $\LIN$ does not.} 
is proper.  The following corollary to the theorem is proved in the same way.  
\begin{cor}\label{haher}
The system $\RCAo+(\kappa_{0}^{3})+\textsf{\textup{IND}}$ proves 
\be
\big[[\dim(\R)\leq 1]+\QFAC^{0,1}\big]\asa \big[[\textup{$\R$ is paracompact}]+\QFAC^{0,1}\big] \asa \LIL. 
\ee
\end{cor}
In conclusion, it is easy to see that the proofs of this section go through for the Lindel\"of lemma for $\R$ based on {any `reasonable' notion} of cover.
Indeed, as long as the formulas `$x\in U_{y}$' and `$[a,b]\subseteq U_{x}$' for the new notion of cover $\cup_{x\in \R}U_{x}$ of $\R$ are decidable in $\Z_{2}^{\Omega}$, the above proofs go through (assuming $(\kappa_{0}^{3})$).  
Since $\Z_{2}^{\Omega}$ can decide if $Y:\R\di \{0,1\}$ represents an open subset of $\R$ (using the textbook definition of open set), this notion of `reasonable' again seems quite reasonable. 

\subsubsection{Partitions of unity}\label{unite}
The notion of \emph{partition of unity} was introduced in 1937 by Dieudonn\'e in \cite{nogeengodsgeschenkje} and this notion is equivalent to paracompactness in a rather general setting by \cite{bengelkoning}*{Theorem 5.1.9}.
We study partitions of unity in this section motivated as follows: on one hand, Simpson proves the existence of {partitions of unity} for complete separable spaces in the proof of \cite{simpson2}*{II.7.2}, i.e.\ this notion has been studied in RM.  

\smallskip

The definition of partition of unity is as follows in Munkres \cite{munkies}*{p.\ 258}
\bdefi
Let $\{U_{\alpha}\}_{\alpha\in J}$ be an indexed open covering of $X$. An indexed family of
continuous functions $\phi_{\alpha}: X \di [0, 1]$ is said to be a partition of unity on $X$, dominated\footnote{Munkres uses `dominated by' in \cite{munkies} instead of Engelking's `subordinate to' in \cite{bengelkoning}.} by $\{U_{\alpha}\}$, if:
\begin{enumerate}
\item $\textsf{support}(\phi_{\alpha})\subset U_{\alpha}$ for each $\alpha\in J$.
\item The indexed family $\{\textsf{support}(\phi_{\alpha})\}_{\alpha\in J}$ is locally finite.
\item $\sum_{\alpha\in J}\phi_{\alpha}(x)=1$ for each $x\in X$.
\end{enumerate}
where $\textsf{support}(f)$ is the closure of the open set $\{x\in X:f(x)\ne 0\}$.
\edefi\noindent
Note that the second item implies that the sum in the third one makes sense.  
With this definition in place, $\PUNI(I)$ is the statement that for any cover generated by $\psi:I\di \R$,  there is a partition of unity $\phi:I\times I\di \R$ of $I$ dominated by $\cup_{x\in I}I_{x}^{\psi}$.
\begin{thm}\label{kokken3}
The system $\RCA_{0}^{\omega}+(\kappa_{0}^{3})+\IND$ proves $[\WKL+\PUNI(I)]\asa \HBT$.
\end{thm}
\begin{proof}
In case of $\neg(\exists^{2})$, the equivalence is easy: all $\R\di \R$-functions are continuous and $\PUNI(I)$ is provable as in the proof of \cite{simpson2}*{II.7.2}, while $\HBT$ follows from $\WKL$ as in the proof of Corollary \ref{dorkeeee}. 
In case of $(\exists^{2})$, the reverse implication is also straightforward: the finite sub-cover provided by $\HBT$ is readily refined, and the existence of a partition 
of unity for a finite cover follows from \cite{simpson2}*{II.7.1}. 

\smallskip

Finally, for the forward direction assuming $(\exists^{2})$, let $\psi:I\di \R$ be as in $\HBT$ and obtain $\phi:I^{2}\di \R$ as in $\PUNI(I)$, i.e.\ for $V_{x}:=\textsf{support}(\phi(x, \cdot))$, 
the open cover $\cup_{x\in I }V_{x}$ of $I$ is locally finite and satisfies $V_{x}\subset I_{x}^{\psi}$.  Now consider:
\be\label{centrifuge2555}\textstyle
(\forall x\in I)(\exists n\in \N)\underline{(\exists y\in I)\big((x-\frac{1}{2^{n}}, x+\frac{1}{2^{n}})\subseteq V_{y}\big)}, 
\ee
Applying $\QFAC^{1,0}$ to \eqref{centrifuge2555}, since $\exists^{3}$ is given, we obtain $\Psi:I\di \R^{+}$ such that $\cup_{x\in I}I_{x}^{\Psi}$ covers $I$. 
Now repeat the proof of Theorem \ref{paramaeremki} for $\Psi$ in place of $\phi$, which  yields $y_{1},\dots y_{k}\in I$ such  $\cup_{i\leq k }I_{x}^{\Psi}$ is a finite sub-cover of $I$.  
Note that in the previous `repeated proof', we do not need the function $\Xi$ (from the proof of Theorem \ref{rathergen}), as $I_{x}^{\Psi}$ covers $x$ for any $x\in I$.  
The aforementioned finite sub-cover of $\cup_{x\in I}I_{x}^{\Psi}$ now yields a finite sub-cover of $\cup_{x\in I}I_{x}^{\psi}$ using $\IND$ in the same way as for Theorem~\ref{fugu}. 
\end{proof}
\begin{cor}\label{kokken322}
The system $\RCA_{0}^{\omega}+\IND+(\kappa_{0}^{3})+\PUNI(I)$ proves $\HBU\asa \HBT$.
\end{cor}
Note that previous base theory in the corollary (and hence the theorem) is conservative over Peano arithmetic by \cite{kohlenbach2}*{Prop.\ 3.12} and the proof of \cite{simpson2}*{II.7.2}.
%
%

\smallskip

Next, we obtain a theorem that brings together a number of different strands from this paper, including \emph{effective paracompactness}, first discussed at the end of Section \ref{diemensie2}.  
In the context of $\RCAo$, we say that `$\phi:\R\di \R$ is an \emph{effective} refinement of $\psi:\R\di \R$' if $(\exists \xi:\R\di\R)(\forall x\in \R)(I_{x}^{\phi}\subseteq I_{\xi(x)}^{\psi})$.  
Effective paracompactness expresses the existence of an effective refinement for any open cover.  Moreover, $\PUNI(\R)$ is the statement that for any cover generated by $\psi:\R\di \R^{2}$,  there is a partition of unity $\phi:\R^{2}\di \R^{2}$ dominated by $\cup_{x\in I}J_{x}^{\psi}$.
 
\begin{thm}\label{kokken327}
The system $\RCA_{0}^{\omega}+(\kappa_{0}^{3})+\IND$ proves the following 
\[
[\PUNI(\R)+\QFAC^{0,1}]\asa \LIL\asa [\textup{$\R$ is effectively paracompact}+\QFAC^{0,1}].
\]
\end{thm}
\begin{proof}
We first prove the first equivalence.
In case of $\neg(\exists^{2})$, the equivalence is easy: all $\R\di \R$-functions are continuous and $\PUNI(\R)$ is provable as in the proof of \cite{simpson2}*{II.7.2}, while $\LIL$ follows by taking the countable sub-cover given by the rationals. 
In case of $(\exists^{2})$, the reverse implication is also straightforward: the countable sub-cover provided by $\LIL$ is readily refined, and the existence of a partition 
of unity for a countable cover follows from \cite{simpson2}*{II.7.1}. 

\smallskip

Finally, for the forward direction assuming $(\exists^{2})$, let $\psi:\R\di \R^{2}$ be as in $\LIL$ and obtain $\phi:\R^{2}\di \R^{2}$ as in $\PUNI(\R)$, i.e.\ for $V_{x}:=\textsf{support}(\phi(x, \cdot))$, 
the open cover $\cup_{x\in \R }V_{x}$ of $\R$ is locally finite and satisfies $V_{x}\subset I_{x}^{\psi}$.  Now consider:
\be\label{centrifuge25}\textstyle
(\forall x\in \R)(\exists n\in \N)\underline{(\exists y\in \R)\big((x-\frac{1}{2^{n}}, x+\frac{1}{2^{n}})\subseteq U_{y}\big)}, 
\ee
Applying $\QFAC^{1,0}$ to \eqref{centrifuge25}, since $\exists^{3}$ is given, we obtain $\Psi:\R\di \R^{+}$ such that $\cup_{x\in \R}I_{x}^{\Psi}$ covers $\R$. 
Now repeat the proof of Theorem \ref{paramaeremki} for $\Psi$ in place of $\phi$ and $\R$ instead of $I$.
Then instead of \eqref{tuigs}, we make use of the following:
\be\label{tuigs1337}
(\forall x\in \R )(\exists n\in \N)(\zeta(n)\geq _{\R}|x|).
\ee
Note that in this `repeated proof', we do not need the choice function $\Xi$ (from the proof of Theorem \ref{rathergen}), as $I_{x}^{\Psi}$ covers $x$ for any $x\in I$.  
Applying $\QFAC^{1,0}$ to \eqref{tuigs1337}, we obtain $\Phi^{0\di 1}$ such that $\cup_{n\in \N}I_{\Phi(n)}^{\Psi}$ is a countable sub-cover of the canonical cover generated by $\Psi$.
This countable sub-cover of $\cup_{x\in I}I_{x}^{\Psi}$ now yields a countable sub-cover of $\cup_{x\in I}I_{x}^{\psi}$ using $\IND$ and $\QFAC^{0,1}$ in the same way as for Theorem \ref{kokken}. 
\end{proof}

\subsection{A better base theory}\label{BBB}
The results in the previous section are not completely satisfactory: while the Axiom of Choice is avoided (as much as possible), the use of $(\kappa_{0}^{3})$ amounts to little more than a trick.  
In this section, we develop a better approach based on the \emph{neighbourhood function principle} $\NFP$ from \cite{troeleke1}, which is used in \cite{dagsamIII}*{\S3} to derive e.g.\ $\HBU$.   
As will become clear, the $\NFP$ principle gives rise to an elegant base theory for the results in the previous section.

\smallskip

We now introduce a fragment of $\NFP$ that is a generalisation of $\QFAC^{1,0}$ to the following formula class.   We always assume that $Y$ has type $(1\times 0)\di 0$.
\bdefi[$C$-formula]
A \emph{basic} $C$-formula has the form $(\exists f\in 2^{\N})(Y(f, n^{0})=0)$.  
A (general) $C$-formula is obtained from basic $C$-formulas via $\wedge, \vee, \di,$ and $ \neg$.
\edefi
Note that $C$-formulas can have parameters besides the number variable, but quantifiers are restricted to $C$ and do not alternate inside a basic $C$-formula.  
The following axiom was first studied in \cite{dagsamVI} as an extension of the results in \cite{samph}.
\bdefi[$C$-$\NFP_{0}$]
For any $C$-formula $A(\sigma^{0^{*}})$, we have
\[
(\forall f^{1})(\exists n^{0})A(\overline{f}n)\di (\exists \Phi^{2})(\forall f^{1})A(\overline{f}\Phi(f)).
\]
\edefi
By the below results, $C$-$\NFP_{0}$ yields a good base theory for the RM of topology.  
There are however other (more conceptual) reasons for adopting this axiom. For instance, we prove $[\HBU+C$-$\NFP_{0}]\asa C\text{-}\WKL$ over $\RCAo+\IND$ in \cite{dagsamVI}*{\S4}, where $C$-$\WKL$ 
is $\WKL$ with tree-elementhood `$\sigma\in T$' given by a $C$-formula.  Kohlenbach has studied similar generalisations of $\WKL$ in \cite{kohlenbach4}.  

\smallskip

We now have the following corollary to Theorem \ref{fugu}.  
\begin{cor}\label{fugucor}
The system $\RCAo+\textsf{\textup{IND}}+C\text{-}\NFP_{0}$ proves $\HBU\asa \HBT$
\end{cor}
\begin{proof}
We show that $C\text{-}\NFP_{0}$ applies to \eqref{centrifuge} and the rest of the proof is identical.  
Now, it is instructive to (equivalently) write \eqref{centrifuge} as follows:
\be\label{cetri}\textstyle
(\forall x\in I)(\exists n\in \N)\underline{(\exists y\in I)(([x](n+1)-\frac{1}{2^{n+1}}, [x](n+1)+\frac{1}{2^{n+1}})\subseteq I_{y}^{\psi})}, 
\ee
where $[x](n)$ is the $n$-th approximation of $x$.  One can then write the underlined formula in \eqref{cetri} as $A(\overline{x}n)$ with only minor abuse of notation.
Note that we need $\exists^{2}$ to convert $A(\overline{x}n)$ into a $C$-formula (using a binary representation for $y$).
\end{proof}
We now have the following corollary to Theorem \ref{kokken}.  
\begin{cor}\label{kokkencor}
The system $\RCA_{0}^{\omega}+\textsf{\textup{IND}}+C\text{-}\NFP_{0}$ proves $[\LIN+\QFAC^{0,1}]\asa \LIL$.
\end{cor}
\begin{proof}
Similar to the proof of Corollary \ref{fugucor}, replace \eqref{centrifuge22} by
\be\label{centrifuge22cor}\textstyle
(\forall x\in \R)(\exists n\in \N){(\exists y\in \R)\big(([x](n+1)-\frac{1}{2^{n+1}}, [x](n+1)+\frac{1}{2^{n+1}})\subseteq J_{y}^{\psi}\big)},
\ee
and note that $C\text{-}\NFP_{0}$ applies.
\end{proof}
In light of the previous corollaries, it is clear one can similarly replace $(\kappa_{0}^{3})$ by $C\text{-}\NFP_{0}$ in 
Theorems \ref{kokken3} and \ref{kokken327} (and corollaries) by considering modifications of \eqref{centrifuge2555} and \eqref{centrifuge25} similar to \eqref{cetri} and \eqref{centrifuge22cor}.  

\smallskip

By the above, $C\text{-}\NFP_{0}$ seems to be an acceptable/usefull generalisation of $\QFAC^{1,0}$ and should be adopted as part of the base theory as well.  
We invite the reader to ponder a similar generalisation of $\Delta_{1}^{0}$-comprehension.  The answer is given in \cite{samph}.

\section{Conclusion}\label{kurtzenhier}
We have studied the higher-order RM of topology, the notions of \emph{dimension} and \emph{paracompactness} in particular. 
Basic theorems regarding the latter turn out to be equivalent to the \emph{Heine-Borel theorem} for uncountable covers, i.e.\ the former are extremely hard to prove (in terms of comprehension axioms).    
A number of nice splittings was obtained, and we have shown that these results do not depend on the exact definition of cover, even in the absence of the axiom of choice.  
In this section, we discuss the foundational implications of our results, esp.\ as they pertain to the \emph{G\"odel hierarchy}.
Now, the latter is a collection of logical systems ordered via consistency strength.  This hierarchy is claimed to capture most systems that are natural or have foundational import, as follows. 
\begin{quote}
\emph{It is striking that a great many foundational theories are linearly ordered by $<$. Of course it is possible to construct pairs of artificial theories which are incomparable under $<$. However, this is not the case for the ``natural'' or non-artificial theories which are usually regarded as significant in the foundations of mathematics.} (\cite{sigohi})
\end{quote}
Burgess and Koellner corroborate this claim in \cite{dontfixwhatistoobroken}*{\S1.5} and \cite{peterpeter}*{\S1.1}.
The G\"odel hierarchy is a central object of study in mathematical logic, as e.g.\ argued by Simpson in \cite{sigohi}*{p.\ 112} or Burgess in \cite{dontfixwhatistoobroken}*{p.\ 40}.  
Precursors to the G\"odel hierarchy may be found in the work of Wang (\cite{wangjoke}) and Bernays (see \cite{theotherguy,puben}).
Friedman (\cite{friedber}) studies the linear nature of the G\"odel hierarchy in detail.  
Moreover, the G\"odel hierarchy exhibits some remarkable \emph{robustness}: we can perform the following modifications and the hierarchy remains largely unchanged:
\begin{enumerate}
 \renewcommand{\theenumi}{\roman{enumi}}
\item Instead of the consistency strength ordering, we can order via inclusion: Simpson claims that inclusion and consistency strength yield the same\footnote{Simpson mentions in \cite{sigohi} the caveat that e.g.\ $\PRA$ and $\WKL_{0}$ have the same first-order strength, but the latter is strictly stronger than the former.} G\"odel hierarchy as depicted in \cite{sigohi}*{Table 1}.  Some exceptional (semi-natural) statements\footnote{There are some examples (predating $\HBU$ and \cite{dagsamIII}) that fall outside of the G\"odel hierarchy \emph{based on inclusion}, like \emph{special cases} of Ramsey's theorem and the axiom of determinacy from set theory (\cites{dsliceke, shoma}).  These are far less natural than e.g.\ Heine-Borel compactness, in our opinion.} do fall outside of the inclusion-based G\"odel hierarchy.\label{kut} 
\item We can replace the systems with their higher-order (eponymous but for the `$\omega$') counterparts.  The higher-order systems are generally conservative over their second-order counterpart for (large parts of) $\L_{2}$.  Hunter's dissertation contains a number of such general results (\cite{hunterphd}*{Ch.\ 2}).\label{lul}
\end{enumerate}
Now, \emph{if} one accepts the modifications (inclusion ordering and higher types) described in the previous two items, \emph{then} an obvious question is where basic topological theorems pertaining to e.g.\ dimension and paracompactness fit within the G\"odel hierarchy.  As depicted in Figure \ref{xxy}, the aforementioned theorems yield a branch that is \emph{completely} independent of the medium range of the G\"odel hierarchy (with the latter based on inclusion).
The same `independence' holds for basic properties of the gauge integral, including many covering lemmas (see \cite{dagsamIII}), as well as for so-called uniform theorems (see \cite{dagsamV}) in which the objects claimed to exist depend on few of the parameters of the theorem.
Some remarks on the technical details concerning Figure \ref{xxy} are as follows. 
\begin{rem}\rm\label{knellen}
First of all, we use a \emph{non-essential} modification of the G\"odel hierarchy, namely involving systems of higher-order arithmetic, like e.g.\ $\ACA_{0}^{\omega}$ instead of $\ACA_{0}$; these systems are (at least) $\Pi_{2}^{1}$-conservative over the associated second-order system (see e.g.\ \cite{yamayamaharehare}*{Theorem 2.2}).  

\smallskip

Secondly, $\Z_{2}^{\Omega}$ is placed \emph{between} the medium and strong range, as the combination of the recursor $\textsf{R}_{2}$ from G\"odel's $T$ and $\exists^{3}$ yields a system stronger than $\Z_{2}^{\Omega}$.  The system $\SIXK$ does not change in the same way.     

\smallskip

Thirdly, while $\HBT$ clearly implies $\WKL$, the paracompactness of the unit interval does not (by the $\ECF$-translation); this is symbolised by the dashed line.  

\smallskip

Fourth, while $\HBT$ and similar statements are \emph{hard} to prove (in terms of comprehension axioms), these theorems (must) have weak first-order strength in light of their provability in intuitionistic topology (see e.g.\ \cite{waaldijkphd, troelstraphd}). 
\end{rem}
The previous remark also establishes that the systems with superscript `$\omega$' deserve to be called the \emph{higher-order counterparts} of the corresponding second-order systems, while $\Z_{2}^{\Omega}$ does not fall into the same category.  
\begin{figure}[H]
\[
\begin{array}{lll}
&\textup{\textbf{strong}} \hspace{1.5cm}& 
\left\{\begin{array}{l}
\textup{large cardinals}\\
\ZFC \\
\textsf{\textup{ZC}} \\
\textup{simple type theory}
\end{array}\right.
\\
&&\\
&&  ~\quad{ {\Z_{2}^{\Omega}}}+\QFAC^{0,1}\\
&&\\
  &&~\quad{ {\Z_{2}^{\Omega}}} \\
&&\\
{ {\left\{\begin{array}{l}
\textup{covering lemmas like $\LIL$}\\
\textup{basic theorems about para-}\\
\textup{compactness and dimension}\\
\textup{for the real numbers $\R$}
\end{array}\right\}}}
&\textup{\textbf{medium}} & 
\left\{\begin{array}{l}
 {\Z}_{2}^{\omega}+ \QFAC^{0,1}\\
\vdots\\
\textup{$\FIVE^{ {\omega}}$ }\\
\textup{$\ATR_{0}^{ {\omega}}$}  \\
\textup{$\ACA_{0}^{ {\omega}}$} \\
\end{array}\right.
\\
&
\\
{ {\left\{\begin{array}{l}
\textup{covering lemmas like $\HBT$}\\
\textup{basic theorems about para-}\\
\textup{compactness and dimension}\\
\textup{for the unit interval $[0,1]$}
\end{array}\right\}}}
&\begin{array}{c}\\\textup{\textbf{weak}}\\ \end{array}& 
\left\{\begin{array}{l}
\WKL_{0}^{ {\omega}} \\
\textup{$\RCA_{0}^{ {\omega}}$} \\
\textup{$\textsf{PRA}$} \\
\textup{$\textsf{EFA}$ } \\
\textup{bounded arithmetic} \\
\end{array}\right.
\\
\end{array}
\]
\caption{The G\"odel hierarchy with a side-branch for the medium range}\label{xxy}
\begin{picture}(250,0)
\put(195,220){ {\vector(-3,-1){140}}}
\put(190,195){ {\vector(-3,-1){90}}}
\put(160,175){{\line(-5,3){20}}}
\put(160,117){ {\vector(-3,-2){53}}}
\multiput(100,70)(5,0){14}{\line(1,0){3}}
\put(167,70){ {\vector(1,0){8}}}
\put(125,100){ {\vector(3,2){50}}}
\put(162,105){{\line(-5,3){20}}}
\put(-35,87){ {\vector(0,1){27}}}
\put(-33,100){ \textup{(+$\QFAC^{0,1}$)}}
\put(36,115){ {\vector(0,-1){25}}}
\put(38,100){ \textup{(+$\WKL$)}}
\put(92, 92){\setlength{\unitlength}{1cm}\qbezier(0,0)(1.5,0.5)(3.6,3.6)}
\put(89,92){ {\vector(-3,-1){3}}}
\end{picture}
\end{figure}
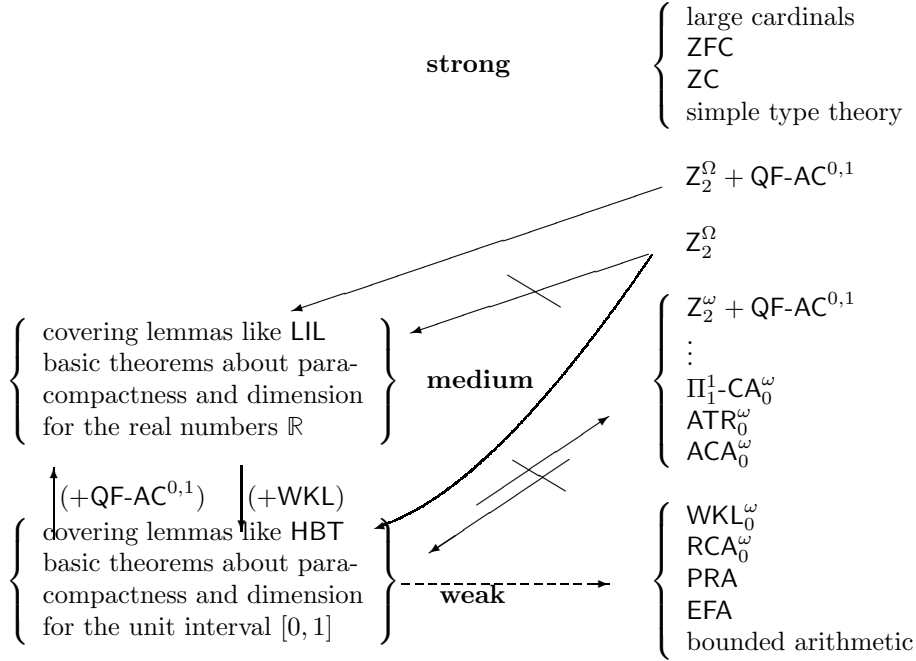~\\
In light of the equivalences involving the gauge integral, uniform theorems, and the Cousin lemma (and hence $\HBT$) from \cite{dagsamIII, dagsamV}, one is tempted to see a serious challenge to the `Big Five' classification from RM,  the linear nature of the G\"odel hierarchy, as well as Feferman's claim that the mathematics necessary for the development of physics can be formalised in relatively weak logical systems (see \cite{dagsamIII}*{p.~24}).
While such an `anti-establishment' view is indeed tempting, a more enlightened interpretation of the aforementioned equivalences can be found in \cite{samph}.
In a nuthshell, the second-order part of the G\"odel hierarchy (including equivalences) is merely the result of applying the $\ECF$-translation to a carefully formulated higher-order hierarchy; this translation maps equivalences to equivalences.
 
 \smallskip

Regarding future work, the following two topics come to mind.  Firstly, there are a number of notions weaker than paracompactness, and it is an interesting question if there are \emph{natural} such notions that yield equivalences with $\HBT$ or weaker theorems.  
Secondly, in light of Remark \ref{diemummy}, it seems interesting to study metrisation theorems in higher-order RM.  We expect that such theorems go far beyond $\SIX$, which features in the second-order RM of topology.  
%
%

\begin{ack}\rm
Our research was supported by the John Templeton Foundation, the Alexander von Humboldt Foundation, LMU Munich (via the Excellence Initiative and the Center for Advanced Studies of LMU), and the University of Oslo.
We express our gratitude towards these institutions. 
We thank Dag Normann for his valuable advice.  
Opinions expressed in this paper do not reflect those of the John Templeton Foundation.    
\end{ack}

\bye